\documentclass[smallextended,numbook,runningheads]{svjour3}     

\smartqed  
\usepackage{graphicx}
\usepackage{amsmath}
\usepackage{amsfonts}
\usepackage{epstopdf} 
\usepackage{epsfig}
\usepackage{bm}
\usepackage{enumitem}
\usepackage{ifpdf}
\usepackage{url}
\usepackage{tabularx}
\usepackage{mathptmx}      

\begin{document}

\title{Efficient high-order two-derivative DIRK methods with optimized phase errors
\thanks{J.O.~Ehigie was supported by the 2024 Tertiary Education Task Fund (TETFUND) National Research Fund (NRF) under grant TETF/ES/DR\&D-CE/NRF2024/CC/EHU/00072 and a grant under IMU-Simons African Fellowship Program which enabled him complete this research at Mississippi State University, USA. V.T.~Luan was supported in part by the National Science Foundation under Grant DMS–2531805.
}}


\author{Julius O. Ehigie         \and
        Vu Thai Luan 
}


\institute{Julius O. Ehigie \at
              Department of Mathematics, University of Lagos, 23401, Nigeria. \\
              \email{jehigie@unilag.edu.ng}           
           \and
           Vu Thai Luan \at
              Department of Mathematics and Statistics, Texas Tech University, 1108 Memorial Circle, Lubbock, TX 79409, USA.\\
              \email{vu.luan@ttu.edu}
}


\maketitle

\begin{abstract}
This work constructs and analyzes new efficient high-order two-derivative diagonally implicit Runge--Kutta (TDDIRK) schemes with optimized phase errors. Specifically, we present a convergence result for TDDIRK methods and investigate their optimized phase errors and linear stability analysis. Based on these, we derive new families of 2-stage fourth-order, 2-stage fifth-order, and 3-stage fifth-order TDDIRK schemes. Finally, we provide numerical experiments at both the ODE and PDE levels to demonstrate the accuracy and efficiency of these new schemes compared to known DIRK schemes in the literature.
\keywords{Two-derivative Runge--Kutta methods \and DIRK methods \and Oscillatory differential equations \and Optimization \and Phase errors}
\subclass{MSC 65L05 \and MSC 65M20}
\end{abstract}

\section{Introduction}\label{sec1}

In this paper, we are concerned with the construction, analysis, and derivation of efficient high-order TDDIRK methods with optimized phase errors for the time integration of PDEs or evolution equations, which, upon spatial discretization, can be formulated as initial value problem (IVP):
\begin{equation}\label{IVP}
y'(t) = f(y(t)), \quad y(t_0)=y_0.
\end{equation}
Here, the solution $y\in \mathbb{R}^N$ is assumed to be sufficiently smooth and  $f: \mathbb{R}^N\rightarrow \mathbb{R}^N$ is sufficiently differentiable and satisfies a local Lipschitz condition along the exact solution. All occurring derivatives are assumed to be uniformly bounded.
By introducing $g(y): = f'(y)f(y) = y''(t)$, the formulation of an $s$-stage TDDIRK method (see e.g., \cite{ChanTsai2010}) is given as follows
\begin{subequations}\label{TDDIRK}
\begin{align}
&Y_i=y_n+c_ihf(y_n)+h^2\sum\limits_{j=1}^ia_{ij}g(Y_j), \qquad i=1,2,\ldots,s \label{InternalTDDIRK}\\
&y_{n+1}=y_n+hf(y_n)+h^2\sum\limits_{i=1}^sb_ig(Y_i)  \label{UpdateTDDIRK}
\end{align}
\end{subequations}
where, $h$ is the time constant stepsize and $a_{ij}$, $b_{i}$ and $c_i$ are the unknown coefficients.
Such schemes belong to the broader class of RK methods, which are among the most widely used integrators for solving \eqref{IVP} due to their simplicity, flexibility, and ease of implementation \cite{Butcher2016,HairerWanner1996,KennedyCarpenter2016}.

For stiff or oscillatory systems, diagonally implicit Runge–Kutta (DIRK) schemes have been particularly attractive since they offer enlarged stability regions while keeping the stage-by-stage computational cost manageable \cite{AlRabeh1993,Alexander1973,BoomZingg2018,KennedyCarpenter2016,SkvortsovKozlov2014}. As a result, DIRK methods remain an active area of research, especially for applications involving large time steps and long-time integration.

In recent years, there has been growing interest in two-derivative Runge–Kutta (TDRK) methods (see, e.g., \cite{Kastlunger_Wanner1972,goeken2000runge}), where both the first and second time derivatives of the solution are incorporated into the scheme \cite{ChanTsai2010,Kalogiratou2013,MonovasilisKalogiratou2021,RanochaDalcinParsaniKetcheson2022}. By exploiting additional derivative information, TDRK schemes are able to achieve higher accuracy and improved dispersion/dissipation properties compared to classical RK methods. High-order explicit TDRK methods up to order seven were analyzed in \cite{ChanTsai2010}, while optimized versions with reduced phase errors have been investigated in several subsequent works \cite{Lele1992,NajafiYazdiMongeau2013,NazariMohammadianZadraCharron2013}.

The combination of these two directions (implicit and multi-derivative formulations) has led to the development of two-derivative DIRK (TDDIRK) schemes \eqref{TDDIRK}. Early constructions of such methods were given in \cite{AhmadSenuIbrahimOthman2019a,KennedyCarpenter2016}, and more recent advances include exponentially fitted TDDIRK and exponential methods for oscillatory problems \cite{EhigieLuanOkunugaYou2022,Luan2024a}, as well as optimized low-dispersion and low-dissipation formulations \cite{CasperShuAtkins1994,Krivovichev2020,RanochaDalcinParsaniKetcheson2022}. These methods extend the advantages of DIRK schemes to the multi-derivative setting, enabling stable and accurate integration of stiff and oscillatory systems.

To our best of knowledge, relatively few works have focused on optimizing phase properties of TDDIRK methods. For oscillatory problems, however, controlling phase errors is essential to ensure accurate long-time integration \cite{HouwenSommeijer1987,Lele1992,TamWebb1993}. While optimized explicit TDRK schemes have been proposed (see, e.g.,\cite{Kalogiratou2013,Krivovichev2020,MonovasilisKalogiratou2021}), their applicability is limited by stringent stability restrictions. Existing implicit approaches either do not systematically address phase accuracy or only treat specific classes such as exponentially fitted methods, where the choice of frequency parameter remains a delicate issue \cite{EhigieLuanOkunugaYou2022,RamosVigoAguiar2010}.

Motivated by these observations,  we present in this paper new families of high-order TDDIRK schemes with optimized phase errors. Specifically, we construct 2-stage fourth-order, 2-stage fifth-order, and 3-stage fifth-order TDDIRK schemes by formulating and solving order conditions under additional optimization criteria targeting phase accuracy and reduced local truncation error. Moreover, a convergence analysis and linear stability study are carried out to confirm their theoretical properties. Finally, numerical experiments on both ODE and PDE test problems demonstrate that the proposed schemes significantly improve accuracy and efficiency compared to existing DIRK and TDDIRK schemes.

The remainder of the paper is organized as follows.
In Section~\ref{Sec_TDDIRK}, we recall the order conditions for TDDIRK methods and provide a convergence analysis . Section~3 develops the optimization framework for controlling phase errors and presents the construction of new families of 2-stage fourth-order and fifth-order and 3-stage fifth-order schemes. Section~\ref{SectionStability} provides the linear stability analysis of the newly constructed schemes. Section~\ref{Sec_experiments} reports numerical experiments that demonstrate the effectiveness of the proposed schemes. Finally, Section~\ref{conclu} concludes the paper with some remarks.

\section{Order conditions and convergence analysis for TDDIRK methods}\label{Sec_TDDIRK}

In this section, we first recall the order conditions for two-derivative DIRK methods. Unlike previous works that focus only on local error analysis (consistency), we provide a global error analysis for these methods.

\subsection{Order conditions}
Let $\hat{e}_{n+1} = \hat{y}_{n+1} - y(t_{n+1})$ denote the local error of the two-derivative DIRK methods \eqref{TDDIRK} at time $t = t_{n+1}$. This error, evaluated in the Euclidean norm $\| \cdot \|$, represents the difference between the true solution $y(t_{n+1})$ and the numerical solution $\hat{y}_{n+1}$, obtained by performing one-step integration using \eqref{TDDIRK}, starting from $y(t_n) =: \tilde{y}_n$,
\begin{subequations}\label{CTDDIRK}
\begin{align}
&\hat{Y}_i=\tilde{y}_n+c_ihf(\tilde{y}_n)+h^2\sum\limits_{j=1}^ia_{ij}g(\hat{Y}_j), \qquad i=1,2,\ldots,s \label{CInternalTDDIRK}\\
&\hat{y}_{n+1}=\tilde{y}_n+hf(\tilde{y}_n)+h^2\sum\limits_{i=1}^sb_ig(\hat{Y}_i).  \label{CUpdateTDDIRK}
\end{align}
\end{subequations}
Following \cite{ChanTsai2010}, the order conditions for \eqref{CTDDIRK} can be derived under  the simplifying assumption
\begin{equation}\label{rowassumption}
\sum\limits_{j=1}^s a_{ij}=\dfrac{c_i^2}{2}, \qquad i=1,2,\ldots,s
\end{equation}
For later use,  we display the order conditions for methods up to order six in Table~\ref{TabOrderConditions}. Note that  $\hat{e}_{n+1} = \mathcal{O}(h^{p+1})$ for methods satisfying the order conditions up to order $p$.

\setlength{\extrarowheight}{1.4 pt}
\renewcommand{\arraystretch}{1.35}
\begin{table}[ht!]
\centering
\caption{Order conditions for TDDIRK methods \eqref{TDDIRK} up to order 6. }
\begin{tabular}{|c|l|c|}
\hline
Order & Order condition(s) & No. \\[2ex]
\hline
2&$\sum\limits_{i=1}^sb_i=\tfrac{1}{2}$ & 1\\[2ex]
\hline
3&$\sum\limits_{i=1}^sb_ic_i= \tfrac{1}{6}$ & 2\\[2ex]
\hline
4&$\sum\limits_{i=1}^sb_ic_i^2=\tfrac{1}{12}$ & 3 \\[2ex]
\hline
5&$\sum\limits_{i,j=1}^sb_ia_{ij}c_j=\tfrac{1}{120}$ & 4 \\[2ex]
&$\sum\limits_{i=1}^sb_ic_i^3=\tfrac{1}{20}$ & 5\\[2ex]
\hline
6&$\sum\limits_{i,j=1}^sb_ia_{ij}c_j^2 =\tfrac{1}{360}$ & 6\\[2ex]
&$\sum\limits_{i,j=1}^sb_ic_ia_{ij}c_j=\tfrac{1}{180}$ & 7 \\[2ex]
&$\sum\limits_{i=1}^sb_ic_i^4=\tfrac{1}{30}$ & 8\\[2ex]
\hline
\end{tabular}\label{TabOrderConditions}
\end{table}

\subsection{Global error analysis}

For the convergence analysis of high-order TDDIRK methods, we rely on the reasonable assumptions on $y(t)$ and $f(y)$ stated in Section 1 to  prove the following lemma first.

 \begin{lemma}\label{lemma2.1}
 Under the assumptions that $y(t) \in \mathbb{R}^N$ is sufficiently smooth and  that  $f(y): \mathbb{R}^N\rightarrow \mathbb{R}^N$ is sufficiently differentiable with bounded derivatives and satisfies a local Lipschitz condition,
the function $g(y)= f'(y)f(y):  \mathbb{R}^N \to  \mathbb{R}^N$ is also locally Lipschitz and has bounded derivatives in $ \mathbb{R}^N$. Consequently, there exists a constant $L_g$ such that
\begin{equation}
\|g(y)-g(z)\| \leq L_g \|y - z \|
\end{equation}
for all $y, \ z \in  \mathbb{R}^N $.
 \end{lemma}
\begin{proof}
The proof is straightforward. Under the given assumptions, it follows that the Jacobian  $f'(y)$ is also locally Lipschitz since
\[
\| f'(y) - f'(z) \| = \| \int_0^1 f''(y + \theta (z - y)) (z - y) \, d\theta \| \leq L_{f'}  \|y - z \|.
\]
Therefore,
\begin{equation*}
\begin{array}{rl}
\|g(y)-g(z)\| &=\|f'(y)f(y)-f'(z)f(z)\|\ =\|(f'(y) - f'(z))f(y) + f'(z)(f(y) - f(z))\| \\[1.5ex]
&\leq L_g \|y - z \|,
\end{array}
\end{equation*}
where the constant $L_g$ depends on the Lipschitz constants $L_{f}$, $ L_{f'}$ of $f(y)$ and $f'(y)$, respectively, as well as their boundedness.
\qed
\end{proof}

Let $e_{n+1}=y_{n+1 }- y(t_{n+1})$ denote the global error, i.e., the difference between the numerical and the exact solution of \eqref{IVP} obtained by the two-derivative DIRK scheme  \eqref{TDDIRK}.

Using Lemma~\ref{lemma2.1}, we establish the following convergence result for TDDIRK methods.

\begin{theorem}\label{theorem2.1}
Under the given assumptions in Lemma~\ref{lemma2.1}, a two-derivative DIRK method \eqref{TDDIRK} applied to  \eqref{IVP}, which fulfills the order conditions of Table~\ref{TabOrderConditions} up to order $p$ ($2 \le p \le 6$),  is convergent of order~$p$ for $h<\dfrac{1}{\max_{1\leq j\leq i}\sqrt{L_g|a_{jj}|}}$ . Namely,  the numerical solution $u_n$ satisfies the global error bound
\begin{equation}\label{eq3.5}
\|e_n\| = \| y_n - y(t_n)\|\leq C h^p
\end{equation}
uniformly on compact time intervals \ $0 \leq  t_n  = nh \leq  T$ with a constant $C$ that depends on $ t_{\text{end}} -t_0$ and the Lipschitz constants of $f(y)$ and $g(y)$.
\end{theorem}
\begin{proof}
It is clear from the definitions of $e_{n}$ and $\hat{e}_{n}$ above that the following relation holds
\begin{equation}\label{GML}
e_{n+1}=(y_{n+1}-\hat{y}_{n+1})+\hat{e}_{n+1}.
\end{equation}
Subtracting \eqref{CUpdateTDDIRK} from \eqref{UpdateTDDIRK} gives
\begin{equation}\label{GMinusL}
 y_{n+1}-\hat{y}_{n+1}=e_n+h[f(y_n)-f(\tilde{y}_n)]+h^2\sum_{i=1}^sb_i[g(Y_i)-g(\hat{Y}_i)].
\end{equation}
Therefore, using the Lipschitz property of $f$ and $g$, we can estimate
\begin{equation}\label{yMinusyHat}
\begin{aligned}
\|y_{n+1}-\hat{y}_{n+1}\|&\leq \|e_n\|+hL_f\| e_n\| +h^2\sum_{i=1}^s|b_i|L_g\|Y_i-\hat{Y}_i\|,\\
&=(1+h L_f)\|e_n\|+ h^2 L_g \sum_{i=1}^s|b_i|\|E_i\|,
\end{aligned}
\end{equation}
where $E_i$ given by
$$
E_i = Y_i-\hat{Y}_i = e_n+c_ih[f(y_n)-f(\tilde{y}_n)]+h^2\sum_{i=1}^i|a_{ij}|[g(Y_j)-g(\hat{Y}_j)]
$$
as the result of subtracting \eqref{CInternalTDDIRK} from \eqref{InternalTDDIRK}.
As in \eqref{yMinusyHat}, one can similarly show that
\begin{equation}\label{ErrorInternal}
\|E_i\|\leq (1+c_ihL_f)\|e_n\|+h^2L_g\sum_{i=1}^i|a_{ij}|\|E_j \|.
\end{equation}
Next, we shall prove by induction that, for $h<\dfrac{1}{\max_{1\leq j\leq i}\sqrt{L_g|a_{jj}|}}$, there exists a constant $C_i>0$ such that $\|E_i \|\leq C_i \|e_n\|$.
Indeed,  for  $i=1$, we have  from \eqref{ErrorInternal},
$$
\begin{array}{rl}
 \|E_1\|\leq (1+c_1hL_f)\|e_n\|+h^2 L_g |a_{11}|  \|E_1\|
\Leftrightarrow (1-h^2L_g |a_{11}|)\|E_1\|\leq (1+c_1hL_f)\|e_n\|.
\end{array}
$$
Therefore, it is true that if $1-h^2L_g |a_{11}| >0$, i.e.,  $h<\dfrac{1}{\sqrt{L_g |a_{11}|}}$, then
$$
\|E_1\|\leq C_1 \|e_n \|, \text{with} \  C_1 = \dfrac{1+c_1hL_f}{1-h^2 L_g |a_{11}|}.
$$
Now, suppose $\|E_j\|\leq C_j\| e_n\|$, for $j=1,2,\ldots,i$, we will show that $\|E_{i+1} \|\leq C_{i+1}\|e_n\|$.
By applying \eqref{ErrorInternal} and using the induction assumption, we obtain
\begin{subequations}\label{ConvergenceGE}
\begin{align}
\nonumber \|E_{i+1}\|&\leq (1+c_{i+1}hL_f)\|e_n\|+h^2L_g\sum_{j=1}^{i+1}|a_{i+1,j}\|E_j\|,\\
\nonumber &=(1+c_{i+1}hL_f)\|e_n\|+ h^2L_g\sum_{j=1}^{i}|a_{i+1,j}\|E_j\| +h^2 L_g |a_{i+1.i+1}| \|E_{i+1}\| ,\\
\nonumber &\leq \big(1+c_{i+1}hL_f +h^2L_g\sum_{j=1}^i|a_{i+1,j}| C_j \big)\|e_n \|+h^2L_g|a_{i+1.i+1}|\|E_{i+1}\|,
\end{align}
\end{subequations}
which is equivalent to
$$
\big(1-h^2 L_g |a_{i+1,i+1}|\big)|\|E_{i+1} \|\leq \big(1+c_{i+1}hL_f +h^2 L_g \sum_{j=1}^i |a_{i+1,j}| C_j \big)\|e_n \|.
$$
Therefore, it is clear that , for  $h<\dfrac{1}{\sqrt{L_g|a_{i+1,i+1}|}}$,  $\|E_{i+1} \|\leq C_{i+1}\| e_n\| $ with
\begin{equation}\label{Ci_formula}
C_{i+1} = \dfrac{1+c_{i+1}hL_f +h^2L_g\sum_{j=1}^i|a_{i+1,j}| C_j}{1-h^2 L_g |a_{i+1,i+1}|}  \rightarrow 1  \ \text{as} \ h   \rightarrow 0.
\end{equation}
From \eqref{GML} and \eqref{yMinusyHat}, we now can estimate
\begin{subequations}\label{ConvergenceGE}
\begin{align}
\nonumber \|e_{n+1} \|&\leq \|y_{n+1}-\hat{y}_{n+1} \|+\| \hat{e}_{n+1}\|,\\
\nonumber &\leq \big(1+hL_f+h^2L_g\sum_{i=1}^s|b_i|C_i\big)\|e_n \|+\|\hat{e}_{n+1}\|.
\end{align}
\end{subequations}

Denoting $M_s=L_f+hL_g\sum_{i=1}^s|b_i|C_i$ (Clearly, $\lim_{h\rightarrow0} M_s=L_f$) leads to
\begin{equation}\label{globalerror}
 \|e_{n+1}\|\leq (1+hM_s)\| e_n\|+\|\hat{e}_{n+1}\|\leq e^{hM_s}\|e_n \|+\| \hat{e}_{n+1}\|.
\end{equation}
Solving \eqref{globalerror} (using the fact that $e_0=0$) gives
\begin{equation}
 \|e_{n+1} \| \leq \sum_{j=0}^n(1+hM_s)^{n-j}\| \hat{e}_{j+1}\|
\leq \sum_{j=0}^ne^{(n-j)hM_s}\| \hat{e}_{j+1}\|
\leq  \sum_{j=0}^ne^{(t_{end}-t_0)M_s}\| \hat{e}_{j+1}\|.
\end{equation}
Since $\hat{e}_{j+1}$ is the local error of the method, it satisfies $\|\hat{e}_{j+1} \|=\mathcal{O}(h^{p+1})$ for methods of order $p$. This implies there exists $\gamma>0$ such that  $\|\hat{e}_{j+1} \| \leq \gamma  h^{p+1}$
Therefore,
\[
\|e_{n+1} \| \leq\sum_{j=0}^ne^{(t_{end}-t_0)M_s} \gamma  h^{p+1} =(n+1)he^{(t_{end}-t_0)M_s} \gamma h^p
 \leq (t_{end}-t_0)e^{(t_{end}-t_0)M_s} \gamma h^p.
\]
\qed
\end{proof}
\section{Construction of new TDDIRK schemes}\label{Construct}

In this section, we recall the concept of dispersion and dissipation errors (see, e.g., \cite{HouwenSommeijer1987})  for constructing high-order TDDIRK schemes with optimized phase and amplification accuracy. In particular, we derive a family of two-stage fourth-order, optimal two-stage fifth-order, and three-stage fifth-order TDDIRK schemes designed to minimize those errors and, when possible, reduce the principal local truncation error. 
\subsection{Dispersion and dissipation} \label{Sec_PhaseProperty}
Consider the oscillatory test problem
\begin{equation}\label{testeq1}
y'={\rm i}\omega y, \quad  \omega>0.
\end{equation}
A numerical one-step method applied to this test equation has the update
\begin{equation}\label{recursivey}
y_{n+1}=R( {\rm i} \nu)y_n, \quad \nu=\omega h,\quad
{\rm i}^2=-1,
\end{equation}
where $ R({\rm i}\nu)$  the is numerical amplification factor (stability function).
\begin{definition}[Dispersion and Dissipation] 
(\cite{HouwenSommeijer1987}) 
With the stability function $R({\rm i}\nu)$, the quantities
\begin{equation}\label{dispdis}
\Psi(\nu)=\nu-{\rm arg}(R({\rm i}\nu)) \ \text{and} \ \
\Phi(\nu)=1-|R({\rm i}\nu)|
\end{equation}
are called the dispersion (phase-lag) and the
dissipation (amplification error), respectively.
The scheme \eqref{TDDIRK} is dispersive of order $p$ and is dissipative of order $q$ if
$$
\Psi(\nu)=\Upsilon_\Psi \nu^{p+1}+\mathcal{O}(\nu^{p+3}), \quad \Phi(\nu)=\Upsilon_\Phi \nu^{q+1}+\mathcal{O}(\nu^{q+3}),
$$
respectively. Also, the constants $\Upsilon_\Psi$ and $\Upsilon_\Phi$ are called the phase-lag and dissipation constants, respectively. In the case $\Psi(\nu)= 0$ or $\Phi(\nu)=0$, it is called zero-dispersive or
zero-dissipative, respectively.
\end{definition}
The phase errors $\Psi(\nu)$ and $\Phi(\nu)$ accumulate during the numerical implementation and are a major source of inaccuracies when large integration steps are used.

Hence, in this work we optimize the coefficients of TDDIRK schemes to increase the dispersion and dissipation orders, or equivalently, to reduce the leading terms in the corresponding dispersion and dissipation errors.

Applying TDDIRK scheme  \eqref{TDDIRK} to  \eqref{testeq1}, one can find $ R({\rm i}\nu)$ as
\begin{equation}\label{Stabfun}
    R({\rm i} \nu)=\big(1-\nu^2 \bm{b}^T(I_s+\nu^2\bm{A})^{-1}\bm{e}) \big)+{\rm i} \big(\nu(1-\nu^2)\bm{b}^T(I_s+\nu^2\bm{A})^{-1}\bm{c}
    \big)
\end{equation}
(here, $I_s$ is the $s \times s$ identity matrix).

Next, we will derive coefficients of 4th and 5th-order TDDIRK schemes, use  \eqref{Stabfun} to obtain their dispersion and the
dissipation \eqref{dispdis}, and optimize them.

Clearly, with $s=1$, it is possible to derive scheme of order 3, however because there is no free parameter to allow for the optimization of the scheme, we start the construction with $s=2$.
\subsection{Optimized two--stage fourth-order schemes}\label{2Stage4OrderConstruct}
For a two-stage ($s=2$), we impose the sufficient conditions in Table \ref{TabOrderConditions} of section \ref{Sec_TDDIRK} for $s=2$.
A two-stage TDDIRK scheme has seven unknown parameters and should satisfy three conditions made up of No. 1--3 in Table \ref{TabOrderConditions} and two row assumptions \eqref{rowassumption} leaving out two free parameters.
Solving the order conditions No. 1--2 for $b_1$ and $b_2$, we obtain
\begin{equation}\label{BcoefficientOrd5}
b_1=\dfrac{1-3c_2}{6(c_1-c_2)}, \qquad b_2=\dfrac{3c_1-1}{6(c_1-c_2)}, \qquad c_1\neq c_2.
\end{equation}
With parameter $b_1$ and $b_2$, fourth order condition is satisfied if
\begin{equation}\label{Order4c1c2Cond}
  2(c_1+c_2-3c_1c_2)-1=0
\end{equation}
for $c_1\neq c_2$. With $c_1=\alpha$ and $a_{21}=\beta$ as a free parameters, the constraint \eqref{Order4c1c2Cond} and the simplifying assumptions \eqref{rowassumption} give the relationships
$$
  c_2=\dfrac{1-2\alpha}{2(1-3\alpha)}. \qquad  a_{11}=\dfrac{\alpha^2}{2} \qquad a_{22}=\dfrac{(1-2\alpha)^2}{8(1-3\alpha)^2}-\beta, \qquad \alpha\neq \dfrac{1}{3}.
$$
The family of two-stage scheme with parameters $\alpha$ and $\beta$ denoted as $\mathtt{TDDIRK4s2(\alpha,\beta)}$ in the Butcher's tableau is given by
\begin{equation}\label{TDDIRK4ButcherTableau}
\centering
\begin{tabular}{c|cc}
  $\alpha$ & $\dfrac{\alpha^2}{2}$ &\\[2ex]
  $\dfrac{1-2\alpha}{2(1-3\alpha)}$ & $\beta$ & $\dfrac{(1-2\alpha)^2}{8(1-3\alpha)^2}-\beta$\\[2ex]
  \hline
 &$\dfrac{1}{6-24\alpha+36\alpha^2}$ & $\dfrac{(1-3\alpha)^2}{3(1-4\alpha+6\alpha^2)}$.
\end{tabular}
\end{equation}
Using Definition 3.1, we now obtain the dispersion and the
dissipation of $\mathtt{TDDIRK4s2(\alpha,\beta)}$ 
\begin{subequations}\label{PhaseConstant}
\begin{align}
&\Psi(\mu)=\frac{3-40\beta(1-3\alpha)^2-4\alpha-10\alpha^2}{240(3\alpha-1)}\nu^5+ \frac{1}{13440(3\alpha-1)^3} D\nu^7 +\mathcal{O}(\nu^9) \label{phi2}\\[2ex]
&\Phi(\mu)=\dfrac{36\alpha^4-72\alpha^3-6\alpha^2+24\alpha-5+72\beta(1-3\alpha)^2(2\alpha^2-4\alpha+1)}{576(1-3\alpha)^2}\nu^6+\mathcal{O}(\nu^8),
\end{align}
\end{subequations}
where
\begin{equation}\label{C72}
\begin{aligned}
D &=
 (90720\beta+2520)\alpha^6-(181440\beta+3360)\alpha^5-2240\beta^2-1120\beta+103\\
& \quad -(181440\beta^2-35280\beta+1820)\alpha^4+(241920\beta^2+84000\beta+824)\alpha^3\\
& \quad -(120960\beta^2+56560\beta-1346)\alpha^2+(26880\beta^2+13440\beta-752)\alpha.
\end{aligned}
\end{equation}

\noindent Next, we seek parameters $\alpha$ and $\beta$ by minimizing the dispersion and dissipation errors such that numerical scheme has the following properties:

\subsubsection*{Property I: Dispersion order 6 and dissipation order 7}
To achieve this property, from \eqref{PhaseConstant}, we solve the system of algebraic equation
\begin{subequations}
\begin{align}
&\dfrac{3-40\beta(1-3\alpha)^2-4\alpha-10\alpha^2}{240(3\alpha-1)}=0 \label{Case1a}\\[2ex]
&\dfrac{36\alpha^4-72\alpha^3-6\alpha^2+24\alpha-5+72\beta(1-3\alpha)^2(2\alpha^2-4\alpha+1)}{576(1-3\alpha)^2}=0 \label{Case1b}
\end{align}
\end{subequations}

\noindent for $\alpha$ and $\beta$. Solving \eqref{Case1a} for $\beta$, we obtain
\begin{equation}\label{beta}
\beta=\dfrac{3-4\alpha-10\alpha^2}{40(1-3\alpha)^2}.
\end{equation}
Substituting $\beta$ in \eqref{Case1b}, $\alpha$ is easily obtained by solving the simplified equation
$$
1-9\alpha+12\alpha^2=0
$$
which yields
$$
\alpha=\dfrac{1}{24}(9\pm \sqrt{33}).
$$
For example, the choice of $\alpha=\dfrac{1}{24}(9- \sqrt{33})$ in \eqref{beta} yields
$$
\beta=\dfrac{23}{960}(1+\sqrt{33})
$$
with the smallest dispersion error constant $|\Upsilon_\Psi|=0.0000062727$. Hence, the two-stage fourth-order scheme can be represented with a Butcher's tableau
\begin{equation*}
\begin{tabular}{c|cc}
  $\dfrac{1}{24}(9-\sqrt{33})$ & $\dfrac{1}{192}(19-3\sqrt{33})$ &\\[2ex]
  $\dfrac{1}{24}(9+\sqrt{33})$ & $\dfrac{23}{960}(1+\sqrt{33})$ & $\dfrac{1}{120}(9-\sqrt{33})$\\[2ex]
  \hline
 &$\dfrac{1}{132}(33+\sqrt{33})$ & $\dfrac{1}{132}(33-\sqrt{33})$
\end{tabular}
\end{equation*}
and shall be denoted $\mathtt{OTDDIRK4s2a}$.

\subsubsection*{Property II. Dispersion order 8 and dissipation order 5}
Similarly to Case I, to obtain the above property, $\Psi(\nu)$ in \eqref{phi2} must vanish up to $\mathcal{O}(\nu^9)$. The elimination of $\beta$ in the resulting equations yields
\begin{equation}\label{Stage3AlphaCondition}
2-20\alpha+35\alpha^2-35\alpha^3=0.
\end{equation}
Solving \eqref{Stage3AlphaCondition} for $\alpha$ yields one real values which is substituted in \eqref{beta} to obtain
$$
\alpha=\dfrac{1}{3}-\dfrac{\left(34300+525\sqrt{6699} \right)^{2/3}-875}{105\left(34300+525\sqrt{6699} \right)^{1/3}},\qquad \beta=\dfrac{3-4\alpha-10\alpha^2}{40(1-3\alpha)^2}.
$$
The corresponding two-stage fourth-order scheme takes the Butcher's form
\begin{equation*}
\begin{array}{cc}
{\scriptsize
\begin{tabular}{c|cc}
  $\alpha$ & $\dfrac{\alpha^2}{2}$ &\\[2ex]
  $\dfrac{1-2\alpha}{2(1-3\alpha)}$ & $\beta$ & $\dfrac{(1-2\alpha)^2}{8(1-3\alpha)^2}-\beta$\\[2ex]
  \hline
 &$\dfrac{1}{6-24\alpha+36\alpha^2}$ & $\dfrac{(1-3\alpha)^2}{3(1-4\alpha+6\alpha^2)}$
\end{tabular},} & \begin{array}{l}
                   \alpha=\dfrac{1}{3}-\tfrac{\left(34300+525\sqrt{6699} \right)^{2/3}-875}{105\left(34300+525\sqrt{6699} \right)^{1/3}}, \\[2ex]
                   \beta=\dfrac{3-4\alpha-10\alpha^2}{40(1-3\alpha)^2}.
                 \end{array}
\end{array}
\end{equation*}
and shall be denoted $\mathtt{OTDDIRK4s2b}$.

\subsection{An optimal two-stage fifth-order scheme}

Here, we seek for a special 2-stage TDDIRK scheme of order 5. This requires that the coefficients of the numerical scheme satisfy the order conditions 1--5 presented in Table \ref{TabOrderConditions}.

As established in Subsection \ref{2Stage4OrderConstruct}, a two-stage scheme of order four has the Butcher tableau \eqref{TDDIRK4ButcherTableau} with parameter $\alpha$ and $\beta$. To extend this to order 5, we derive unique parameters $\alpha$ and $\beta$ such that numerical scheme additionally satisfies two extra order conditions No. 4 and 5 in Table \ref{TabOrderConditions}. This requirements leads to the following system of equations
\begin{subequations}
\begin{align}
&\dfrac{\alpha(2-48\beta)+8\beta+\alpha^2(2+72\beta)-1}{48(3\alpha-1)}=\dfrac{1}{120}, \label{Ord5a}\\
&\dfrac{2\alpha^2+2\alpha-1}{72\alpha-24}=\dfrac{1}{20} \label{Ord5b}.
\end{align}
\end{subequations}
Solving \eqref{Ord5b} for $\alpha$ and substituting into \eqref{Ord5a}, yields the parameters
$$
\alpha=\dfrac{1}{10}(4 - \sqrt{6}),\qquad  \beta=\dfrac{1}{50}(2+3\sqrt{6}),
$$
which defines a unique two-stage, fifth order scheme. This scheme has the smallest phase-lag constant, $|\Upsilon_\Psi|=0.000173639$ and represented in the Butcher's tableau
\begin{equation*}
\begin{tabular}{c|cc}
  $\dfrac{1}{10}(4-\sqrt{6})$ & $\dfrac{1}{100}(11-4\sqrt{6})$ &\\[2ex]
  $\dfrac{1}{10}(4+\sqrt{6})$ & $\dfrac{1}{50}(2+3\sqrt{6})$ & $\dfrac{1}{100}(7- 2\sqrt{6})$\\[2ex]
  \hline
  &$\dfrac{1}{36}(9+\sqrt{6})$ & $\dfrac{1}{36}(9-\sqrt{6})$
\end{tabular}.
\end{equation*}
We shall refer to this method as $\mathtt{TDDIRK5s2}$.

\subsection{Optimized three--stage fifth-order schemes}

For a 3--stage scheme of order 5, the order conditions for this
case ($s=3$) must satisfy order conditions 1--5 in Table \ref{TabOrderConditions}.
Given that the three-stage TDDIRK scheme has 12 coefficients, applying five order conditions together with the three row assumptions \eqref{rowassumption} leaves four free parameter. For simplicity, we set $c_1=0$, so that the resulting TDDIRK scheme has an explicit first stage.

The weights are determined by solving order conditions 1--3 in Table \ref{TabOrderConditions} to obtain
$$
b_1=\dfrac{1-2(c_2+c_3-3c_2c_3)}{12c_2c_3}, \qquad b_2=\dfrac{1-2c_3}{12c_2(c_2-c_3)}, \qquad b_3=\dfrac{2c_2-1}{12c_3(c_2-c_3)}.
$$
The corresponding row assumptions for a 3-stage scheme gives the relationships
$$
a_{11}=0,\qquad a_{22}=\dfrac{c_2^2}{2}-a_{21}, \qquad a_{33}=\dfrac{c_3^2}{2}-a_{31}-a_{32}.
$$
Using the weights $b_i$ and consequences of the row assumptions,
we solve for $c_3$ and $a_{32}$ using order conditions 4--5 in Table \ref{TabOrderConditions}, to obtain
$$
c_3=\dfrac{3-5c_2}{5(1-2c_2)}, \qquad
a_{32}=\dfrac{2-240b_2a_{22}c_2-120b_3c_3^3+240b_3a_{31}c_3}{240b_3(c_2-c_3)}.
$$
Thus, the coefficients $a_{21}$, $a_{31}$ and $c_2$ can be chosen such that the scheme is optimized by minimizing the dispersion, dissipation errors as well as the local error $\|\hat{e}_{n+1}\|$.\\

\noindent The Taylor expansion of quantities $\Psi(\nu)$ and $\Phi(\nu)$ are given by
\begin{subequations}\label{PsiPhi3}
\begin{align}
&\Psi(\nu)=\dfrac{E}{420(2c_2-1)^2(3-10c_2+10c_2^2)}\nu^7+\mathcal{O}(\nu^{9}), \label{Psi33} \\
&\Phi(\nu)=\dfrac{F}{6(2c_2-1)(3-10c_2+10c_2^2)}\nu^6+ \mathcal{O}(\nu^{8}), \label{Phi3}
\end{align}
\end{subequations}
where
\begin{equation}\label{E}
\begin{aligned}
E &=
 25 (280 a_{31}-61) c_2 - 1000 (371 a_{31}-33) c_2^4
+(7995 - 59500 a_{31}) c_2^2 \\
& \quad + 50 (4130 a_{31}-431) c_2^3
+14000 (25 a_{31}-2) c_2^5 +3000a_{31}c_2^2-220c_2^2
 \\
& \quad + 70 a_{21} (9 + (20 - 500 a_{31}) c_2
+(40 a_{31}-3)(100c_2^4-150c_2^3-3500c_2^6))\\
& \quad +1400 a_{21}^2 (5 c_2-3)+102 \hspace{2in}
\end{aligned}
\end{equation}
and
\begin{equation}\label{F}
\begin{aligned}
F &=
12 + 90 a_{21} - 145 c_2 - 150 a_{21} c_2 + 750 a_{31} c_2 + 525 c_2^2 + 4500 a_{31} c_2^2 - 800 c_2^3 \\
& \quad  + 9000 a_{31} c_2^3 + 450 c_2^4 - 6000 a_{31} c_2^4. \hspace{2in}
\end{aligned}
\end{equation}

\noindent We optimize the scheme by allowing $E$ and $F$ to vanish in \eqref{PsiPhi3}, thereby obtaining a numerical scheme with dispersion order 8 and dissipation order 7, respectively. The quantities $E$ and $F$ vanish for
$$
a_{21}=\dfrac{1}{70}(2-11c_2+35c_2^2), \qquad a_{31}=\dfrac{1575c_2^3-2275c_2^2+940c_2-102}{5250c_2(1-2c_2)^2}.
$$
The parameter $c_2$ shall be obtained by a numerical search for $c_2$ where $\|\hat{e}_{n+1}\|$ is minimum. Specifically, $\|\hat{e}_{n+1}\|$ for a fifth-order scheme is derived from order six conditions 6--8 in Table \ref{TabOrderConditions}. Substituting the derived parameters, $\|\hat{e}_{n+1}\|$, depending on $c_2$, is given by
\begin{equation}\label{LTE3}
\|\hat{e}_{n+1} \|=\frac{122 - 2055 c_2 + 12930 c_2^2 - 40225 c_2^3 + 66150 c_2^4 - 55125 c_2^5 +
 18375 c_2^6}{3969000000 (2 c_2-1)^3}.
\end{equation}
Following \cite{SofroniouSpaletta2004}, a necessary condition for a minimum can be derived by setting the derivative to zero, that is
$$
\dfrac{d\|\hat{e}_{n+1}  \|}{dc_2}=\dfrac{(1-5c_2+5c_2^2)^2(3-10c_2+10c_2^2)}{9000000(1-2c_2)^4}=0.
$$
Solving the resulting equations for parameter $c_2$, the real solutions are given by
$$
c_2^{\text{*}}=\dfrac{1}{10}(5\pm \sqrt{5}).
$$
Further more, since $\dfrac{d^2}{dc_2^2}\| \hat{e}_{n+1}\|(c_2^{\text{*}})=0$, this implies an inflection at $c_2^{\text{*}}$. The choice $c_2=\dfrac{1}{10}(5-\sqrt{5})$ yields the smallest phase-lag constant $|\Upsilon_\Psi|=0.000000449669$ and the corresponding scheme in Butcher tableau is
\begin{equation*}
\begin{tabular}{c|ccc}
  0 & 0  &  &\\[2ex]
  $\dfrac{1}{10}(5-\sqrt{5})$ & $\dfrac{1}{10}-\dfrac{6}{175}\sqrt{5}$ & $\dfrac{1}{20}-\dfrac{11}{700}\sqrt{5}$  &
  \\[2ex]
  $\dfrac{1}{10}(5+\sqrt{5})$ & $\dfrac{20+19\sqrt{5}}{1050}$ & $\dfrac{17}{1050}(5+3\sqrt{5})$  &  $\dfrac{1}{60}(3-\sqrt{5})$
  \\[2ex]
  \hline
& $\dfrac{1}{12}$ & $\dfrac{1}{24}(5+\sqrt{5})$ & $\dfrac{5}{6(5+\sqrt{5})}$
\end{tabular}.
\end{equation*}
This scheme shall be denoted as $\mathtt{OTDDIRK5s3}$.

Putting altogether, we now summarize in Table \ref{TablePhaseProperty}, the dispersion and dissipation errors for the newly derived TDDIRK schemes with optimized phase errors.
\begin{table}[ht!]
\begin{center}
\begin{minipage}{174pt}
\caption{Dispersion and dissipation of the optimized TDDIRK schemes}\label{TablePhaseProperty}\label{TablePhaseProperty}%
\begin{tabular}{ll}
\hline\noalign{\smallskip}
Method & Dispersion ($|\Psi(\nu)|$)\\
       & Dissipation ($|\Phi(\nu)|$)\\
\hline\noalign{\smallskip}
$\mathtt{OTDDIRK4s2a}$ & $0.0000062727 \nu^7 +\mathcal{O}(\nu^9)$\\
              & $0.0000474716 \nu^8 +\mathcal{O}(\nu^{10})$\\[2ex] $\mathtt{OTDDIRK4s2b}$ & $0.00001112846 \nu^9 + \mathcal{O}(\nu^{11})$\\
              & $0.00007992350 \nu^6 +\mathcal{O}(\nu^8)$\\[2ex]
$\mathtt{TDDIRK5s2}$ & $0.000173639 \nu^7 + \mathcal{O}(\nu^{9})$\\
              & $0.000138889 \nu^6 +\mathcal{O}(\nu^8)$\\[2ex]
$\mathtt{OTDDIRK5s3}$  & $0.000000449669 \nu^9+\mathcal{O}(\nu^{11})$ \\
                       & $0.000000563909\nu^8+\mathcal{O}(\nu^{10})$\\
\hline\noalign{\smallskip}
\end{tabular}
\end{minipage}
\end{center}
\end{table}
\section{Stability regions of the newly derived TDDIRK schemes}\label{SectionStability}
This section is concerned with the linear stability analysis
of the TDDIRK methods derived in Section \ref{Construct}. The linear stability is analyzed by applying the method \eqref{TDDIRK} to the scalar test equation
\begin{equation}\label{testeq}y'=\lambda y, \quad
\quad  \lambda>0.
\end{equation}
This results in the following difference equation
$$
y_{n+1}=R(z)y_n, \qquad z=\lambda h.
$$
The stability function (or growth function) is given by
\begin{equation}
\dfrac{y_{n+1}}{y_n}=R(z)=1+z+z^2 \bm{b} \cdot \big( (I-z^2\bm{A})^{-1}(\bm{e}+ \bm{c}z)\big).
\end{equation}
where $I$ is the $s\times s$ identity matrix. The method \eqref{TDDIRK} is absolutely stable if $|R(z)|\leq 1$.

Figure \ref{FigStabTDDIRK} presents the stability plots of TDDIRK schemes derived in Section \ref{Construct} on the plane specified by $[-6,0.3] \times [-6,6]$ which includes the stability region of the newly derived schemes $\mathtt{OTDDIRK4s2a}$, $\mathtt{OTDDIRK4s2b}$, $\mathtt{TDDIRK5s2}$, $\mathtt{OTDDIRK5s3}$ compared with existing TDDIRK schemes ($\mathtt{TDDIRK4s2}$, $\mathtt{TDDIRK5s3}$) presented in \cite{AhmadSenuIbrahimOthman2019a}. It can be observed that the newly optimized TDDIRK schemes possess larger stability region than the existing TDDIRK schemes presented in \cite{AhmadSenuIbrahimOthman2019a}.
\begin{figure}[ht!]
\centering
\includegraphics[width=12cm,height=6.5cm]{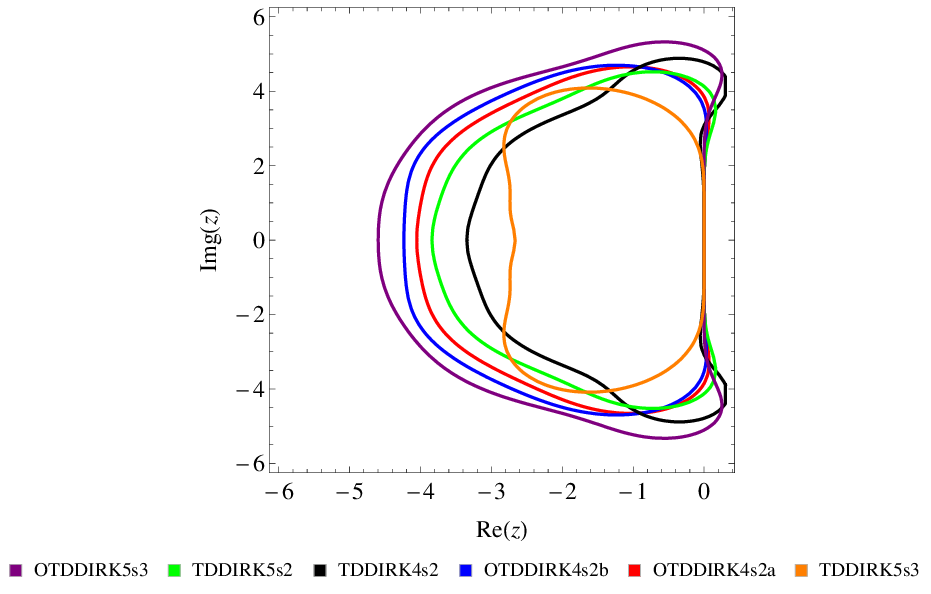}
\caption{Stability plots for TDDIRK schemes.}\label{FigStabTDDIRK}
\end{figure}

\section{Numerical experiments}\label{Sec_experiments}

In order to assess the effectiveness of the newly constructed TDDIRK schemes derived in Section \ref{Construct}, we run numerical tests to problems, such as 2D harmonic oscillator \cite{Kalogiratou2013}, an advection equation with a nonlinear source term \cite{GriffithsStuartYee1992} and a 2D advection-diffusion-reaction problem \cite{Luan2013}. We present here the performance in terms of the accuracy and CPU time of the schemes when compared to existing optimized DIRK and TDDIRK schemes of the same orders in the literature.

For accuracy comparison, same step sizes will be used for all the considered schemes. For efficiency comparison, the step sizes will be chosen so that these schemes achieve about the same error thresholds. The errors at the final time are computed using the maximum norm ($\log_{10}(MGE)$). The internal stages of the TDDIRK schemes are sequentially computed by using the  fixed point iteration technique and the stoping criterion for the iterative procedure is given by
\begin{equation*}
\|Y_i^{(r+1)}-Y_i^{(r)} \|_2 <{10}^{-12}, \qquad r=0,1,2,\ldots,
\end{equation*}
where $Y_i^{(r)}$, is the value at the $r$th iteration in the iterative process.

The new fourth-order schemes ($\texttt{OTDDIRK4s2a}$ and $\texttt{OTDDIRK4s2b}$) are compared with: the 3-stage DIRK schemes ($\mathtt{ILDDRK4}$) \cite{NajafiYazdiMongeau2013}, 4-stage DIRK schemes ($\mathtt{LDDDIRK4}$) \cite{NazariMohammadianCharron2015}, a 4-stage singly DIRK scheme ($\mathtt{SDIRK4s4}$) \cite{BoomZingg2018}, 7-stage ESDIRK scheme \\($\mathtt{ESDIRK4s7}$) \cite{KennedyCarpenter2019}, 2-stage TDDIRK scheme ($\mathtt{TDDIRK4s2}$) \cite{AhmadSenuIbrahimOthman2019a}, respectively.

Further more, the new fifth order schemes ($\mathtt{TDDIRK5s2}$ and $\mathtt{OTDDIRK5s3}$) are similarly compared with: the 5-stage EDIRK scheme ($\mathtt{EDIRK5s5}$) \cite{Kraaijevanger1991}, 4-stage EDIRK scheme ($\mathtt{EDIRK5s4}$) \cite{AlRabeh1993}, 5-stage SDIRK scheme ($\mathtt{SDIRK5s5}$) \cite{HairerWanner1996}, 5-stage ESDIRK scheme ($\mathtt{ESDIRK5s5}$) \cite{SkvortsovKozlov2014}, 7-stage ESDIRK scheme ($\mathtt{ESDIRK5s7}$) \cite{KennedyCarpenter2019} and the 3-stage TDDIRK scheme ($\mathtt{TDDIRK5s3}$) \cite{AhmadSenuIbrahimOthman2019a}, respectively.
\subsection{The 2D harmonic oscillator}\label{Harmonic}
\noindent First, we consider the two dimensional harmonic oscillator  \cite{Kalogiratou2013} described by the Hamiltonian
$$
H(p_1,p_2,q_1,q_2)=\dfrac{1}{2}(p_1^2+P_2^2)+\dfrac{1}{2}(q_1^2+q_2^2).
$$
The equations of motion are
$$
p'_1=-q_1, \qquad q'_1=p_1, \qquad p'_2=-q_2, \qquad q'_2=p_2
$$
with initial conditions
$$
p_1(0)=0\qquad q_1(0)=1, \qquad p_2(0)=1, \qquad q_2(0)=0.
$$
The exact solution of this problem is $q_1(x)=\cos x$ and $q_2(x)=\sin x$.

This problem was integrated on the interval $[0,100]$ with time steps $h=\{\tfrac{1}{4},\tfrac{1}{8},\tfrac{1}{16},\tfrac{1}{32}\}$. The numerical behaviour of the global error norm are presented in Figures \ref{Fig_Harmonic_4} and \ref{Fig_Harmonic_5}.
\begin{figure}
\centering
\begin{tabular}{cc}
\epsfig{file=Prob_Accuracy_Harmonic_4.eps,width=0.47\linewidth,clip=}
\hspace{2mm}
\epsfig{file=Prob_Efficiency_Harmonic_4.eps,width=0.47\linewidth,clip=}
\end{tabular}
\caption{Accuracy (left) and efficiency (right) plots of 4th-order schemes for Example~5.1.}\label{Fig_Harmonic_4}
\end{figure}
\begin{figure}
\centering
\begin{tabular}{cc}
\epsfig{file=Prob_Accuracy_Harmonic_5.eps,width=0.47\linewidth,clip=}
\hspace{2mm}
\epsfig{file=Prob_Efficiency_Harmonic_5.eps,width=0.47\linewidth,clip=}
\end{tabular}
\caption{Accuracy (left) and efficiency (right) plots of 5th-order schemes for Example~5.1.}\label{Fig_Harmonic_5}
\end{figure}
\subsection{Advection with a nonlinear source term}\label{Advection}

Next, we consider the stiff nonlinear advection equation  
\begin{equation}
u_t+u_x=f(u)
\end{equation}
studied in \cite{GriffithsStuartYee1992}, where function $f(u)$ is assumed to satisfy the following conditions:
\renewcommand{\theenumi}{\roman{enumi}}%
\begin{enumerate}
  \item $f(u)\in \mathcal{C}^2([0,1],R)$;
  \item $f(0)=f(1)=0$ and $f(u)>0$ for $u\in(0,1)$;
  \item $f'(0)>0$, $f'(1)<0$
\end{enumerate}
A typical example given satisfying i--iii is $f(u)=u-u^2$. Such problems arise as models of nonequilibrium gas dynamics and, in particular, in the study of transatmospheric vehicles. We consider the problem with a wave of amplitude of $u(t=0,2<x<4)$, imposed as the initial condition.

To discretize the space variables, we consider the upwind scheme with a fine mesh of $\Delta x=0.05$. We compute and compare the errors of numerical solution of the resulting system of ODEs in order to attain four CFL numbers of 0.05, 0.1, 0.2 and 0.4. Figures \ref{Fig_Advection_4} and \ref{Fig_Advection_5} show the maximum of the absolute errors over the integration interval measured in the $l_\infty$ norm at $t=1.4$ for the considered schemes and the efficiency in terms of CPU time, respectively.  Table \ref{TabProb52_4} and \ref{TabProb52_5} present the maximum global error for schemes of order 4 and 5, respectively. The implementation is carried out with time step $h=0.02$ and different meshsize using the uniform grid ($\Delta x=2/N$), $N=50$, 100 and 200 for so that the computation satisfy the CFL condition of 0.5, 1.00 and 2.00 in order to investigate the spatial grid effect.
\begin{figure}[ht!]
\centering
\begin{tabular}{cc}
\epsfig{file=Prob_Accuracy_Advection_4.eps,width=0.47\linewidth,clip=}
\hspace{2mm}
\epsfig{file=Prob_Efficiency_Advection_4.eps,width=0.47\linewidth,clip=}
\end{tabular}
\caption{Accuracy (left) and efficiency (right) plots of 4th-order schemes for Example~5.2.}\label{Fig_Advection_4}
\end{figure}
\begin{figure}[ht!]
\centering
\begin{tabular}{cc}
\epsfig{file=Prob_Accuracy_Advection_5.eps,width=0.47\linewidth,clip=}
\hspace{2mm}
\epsfig{file=Prob_Efficiency_Advection_5.eps,width=0.47\linewidth,clip=}
\end{tabular}
\caption{Accuracy (left) and efficiency (right) plots of 5th-order schemes for Example~5.2.}\label{Fig_Advection_5}
\end{figure}
\begin{table}
  \begin{center}
  \begin{minipage}{174pt}
  \caption{Maximum global error for fourth-order schemes when applied to Example~5.2}\label{TabProb52_4}
  \begin{tabular}{lccc}
\hline\noalign{\smallskip}
      Method & $N=50$ & $N=100$ & $N=200$ \\
\hline\noalign{\smallskip}
     ILDDRK4 & $8.70 \cdot 10^{-3}$ & $6.11 \cdot 10^{-2}$ & $2.95 \cdot 10^{-1}$  \\
     LDDDIRK4 & $1.87 \cdot 10^{-2}$ & $2.73 \cdot 10^{-1}$ & NaN  \\
     SDIRK4s4 & $8.10 \cdot 10^{-3}$ & $1.26 \cdot 10^{-1}$ & NaN \\
     ESDIRK4s7 & $6.00\cdot 10^{-5}$ & $1.40\cdot 10^{-3}$ & $2.96\cdot 10^{-2}$  \\
     TDDIRK4s2 & $4.00\cdot 10^{-4}$ & $1.00\cdot 10^{-2}$ & $3.88\cdot 10^{-1}$  \\
     OTDDIRK4s2a & $7.87\cdot 10^{-6}$ & $3.89\cdot 10^{-4}$ & $2.97\cdot 10^{-2}$  \\
     OTDDIRK4s2b & $1.74\cdot 10^{-5}$ & $6.85\cdot 10^{-4}$ & $3.41\cdot 10^{-2}$  \\
\hline\noalign{\smallskip}
   \end{tabular}
    \end{minipage}
  \end{center}
\end{table}
\begin{table}
  \begin{center}
  \begin{minipage}{174pt}
  \caption{Maximum global error for fifth-order schemes when applied to Example~5.2}\label{TabProb52_5}
  \begin{tabular}{lccc}
\hline\noalign{\smallskip}
        Method & $N=50$ & $N=100$ & $N=200$  \\
\hline\noalign{\smallskip}
     EDIRK5s5 & $9.51\cdot 10^{-6}$ & $4.69\cdot 10^{-4}$ & $1.95\cdot 10^{-2}$  \\
       EDIRK5s4 & $1.20\cdot 10^{-3}$ & $4.00\cdot 10^{-3}$ & $2.34\cdot 10^{-2}$ \\
     SDIRK5s5 & $2.09 \cdot 10^{-4}$ & $8.30 \cdot 10^{-3}$ & $1.61 \cdot 10^{-1}$  \\
     ESDIRK5s5 & $3.61\cdot 10^{-4}$ & $8.50\cdot 10^{-3}$ & $1.56\cdot 10^{-1}$ \\
     ESDIRK5s7 & $7.61\cdot 10^{-6}$ & $3.45\cdot 10^{-4}$ & $1.29\cdot 10^{-2}$  \\
     TDDIRK5s3 & $5.44\cdot 10^{-5}$ & $2.80\cdot 10^{-3}$ & $2.16\cdot 10^{-1}$  \\
     TDDIRK5s2 & $2.10\cdot 10^{-5}$ & $1.00\cdot 10^{-3}$ & $4.68\cdot 10^{-2}$ \\
     OTDDIRK5s3 & $1.86\cdot 10^{-7}$ & $3.20\cdot 10^{-5}$ & $5.50\cdot 10^{-3}$ \\
\hline\noalign{\smallskip}
   \end{tabular}
  \end{minipage}
  \end{center}
\end{table}
\subsection{2D advection-diffusion-reaction equation }\label{AdvectionDiffusion2D}

Finally, we consider the two-dimensional advection-diffusion-reaction equation (see \cite{Luan2013})
$$
\partial_t u=\epsilon \Delta u-\alpha (\nabla \cdot u)+\gamma u (u-\frac{1}{2})(1-u)
$$
with initial values
$$
u(x,y,0)=0.3+256\left(x(1-x)y(1-y)\right)^2
$$
on the unit square boundary $\Omega=[0,1]^2$, subject to the homogeneous Neumann boundary conditions with parameters $\epsilon=1/100$, $\alpha=-10$, $\gamma=100$. The solution space is discretized using the standard finite difference using 101 grid points in each direction with meshsize $\Delta x=\Delta y=1/20$.

We integrate the resulting system of ODEs and compare numerical results at $t=0.08$. The expected order plots and efficiency curves are displayed in Figures \ref{Fig_Advection_Diffusion_4} and \ref{Fig_Advection_Diffusion_5}.
\begin{figure}[ht!]
\centering
\begin{tabular}{cc}
\epsfig{file=Prob_Accuracy_Advection_Diffusion_Reaction_4.eps,width=0.47\linewidth,clip=}
\hspace{2mm}
\epsfig{file=Prob_Efficiency_Advection_Diffusion_Reaction_4.eps,width=0.47\linewidth,clip=}
\end{tabular}
\caption{Accuracy (left) and efficiency (right) plots of 4th-order schemes for Example~5.4.}\label{Fig_Advection_Diffusion_4}
\end{figure}
\begin{figure}[ht!]
\centering
\begin{tabular}{cc}
\epsfig{file=Prob_Accuracy_Advection_Diffusion_Reaction_5,width=0.47\linewidth,clip=}
\hspace{2mm}
\epsfig{file=Prob_Efficiency_Advection_Diffusion_Reaction_5.eps,width=0.47\linewidth,clip=}
\end{tabular}
\caption{Accuracy (left) and efficiency (right) plots of 5th-order schemes for Example~5.4.}\label{Fig_Advection_Diffusion_5}
\end{figure}

The numerical test problems \ref{Harmonic} and  \ref{Advection} fall within the class of oscillatory systems, while problem \ref{AdvectionDiffusion2D} represents a highly stiff case.

For oscillatory Problems \ref{Harmonic}--\ref{Advection}, the optimized fourth-order TDDIRK schemes are consistently superior to the reference schemes, both in terms of accuracy and efficiency for all the problems. Among the fifth-order schemes, $\mathtt{OTDDIRK5s3}$ is the most effective. In contrast, this is not true for the two-stage fifth-order scheme, $\mathtt{TDDIRK5s2}$. Observe from Figures \ref{Fig_Harmonic_5} and \ref{Fig_Advection_5} that the EDIRK5s5 and ESDIRK5s7 schemes achieves a higher accuracy than $\mathtt{TDDIRK5s2}$. This may be due to the fact that it is not optimized. Nevertheless, aside from $\mathtt{OTDDIRK5s3}$, the scheme $\mathtt{TDDIRK5s2}$ is most efficient among the fifth-order schemes considered. Above all, $\mathtt{OTDDIRK5s3}$ is the most accurate and efficient fifth-order scheme overall. Furthermore, for very small time steps, $\mathtt{OTDDIRK5s3}$ yields an error threshold of about $10^{-12}$ highlighting its stability and accuracy at high resolutions.

Turning to the stiff Problem \ref{AdvectionDiffusion2D}, Figure \ref{Fig_Advection_Diffusion_4} shows that optimized fourth-order TDDIRK scheme $\mathtt{OTDDIRK4s2b}$ is slightly more accurate than the $\mathtt{ESDIRK4s7}$. Although both methods perform similarly in terms of accuracy, $\mathtt{OTDDIRK4s2b}$ is notably more efficient. However, we observe that the $\mathtt{OTDDIRK4s2a}$ scheme clearly dominates all fourth-order schemes under comparison in terms of accuracy and efficiency.

For the fifth-order schemes applied to the stiff problem, Figure \ref{Fig_Advection_Diffusion_5} shows that the $\mathtt{TDDIRK5s2}$ achieves a comparable accuracy with $\mathtt{EDIRK5s5}$ and $\mathtt{ESDIRK5s7}$. This contrasts with the oscillatory test problems, where $\mathtt{EDIRK5s5}$ and $\mathtt{ESDIRK5s7}$ schemes are clearly superior. Most importantly, $\mathtt{OTDDIRK5s3}$ is the most efficient schemes in this context as well.

In summary, for all the problems considered, the fourth-order scheme $\mathtt{OTDDIRK4s2a}$ is more accurate and efficient than $\mathtt{OTDDIRK4s2b}$. For the fifth-order schemes, the optimized $\mathtt{OTDDIRK5s3}$ demonstrates superior accuracy compared with $\mathtt{TDDIRK5s2}$, while both schemes exhibit comparable efficiency.

\section{Conclusion}\label{conclu}

%

In this paper, we developed and analyzed new families of two-derivative diagonally implicit Runge--Kutta (TDDIRK) methods with optimized phase accuracy. After restating the order conditions and establishing  convergence theory, we constructed 2-stage fourth-and fifth-order, and 3-stage fifth-order schemes by systematically tuning free parameters to minimize phase errors and reduce principal local truncation error. A linear stability analysis confirmed their robustness, while a detailed study of phase properties demonstrated their superior accuracy compared with existing DIRK and TDDIRK schemes. Numerical experiments on oscillatory ODEs and semi-discrete PDEs further validated the effectiveness of the proposed schemes, showing clear advantages for long-time integration of highly oscillatory solutions and smooth hyperbolic problems.

This work shows the potential of two-derivative implicit integrators in efficiently handling oscillatory and stiff systems. Future research will focus on extending the framework to parallel TDDIRK schemes and investigating symmetric two-derivative methods for applications in nonlinear wave propagation, including elastic and acoustic models.

\section*{Acknowledgements}
Vu Thai Luan would like to thank the Vietnam Institute for Advanced Study in Mathematics (VIASM) for their hospitality during the summer research stay, where part of this work was carried out.

\bibliographystyle{spmpsci}      
\bibliography{Opt2DIRK1}   


\end{document}